\documentclass[a4paper,12pt,twoside,leqno,final]{amsart}
\usepackage{amsmath}
\usepackage{amssymb}

\newtheorem{thm}{Theorem}[section]
\newtheorem{lem}[thm]{Lemma}

\newtheorem{cor}[thm]{Corollary}

\newcommand{\C}{{\mathbb C}}
\newcommand{\D}{{\mathbb D}}

\newcommand{\T}{{\mathbb T}}

\newcommand{\N}{{\mathbb N}}
\newcommand{\cN}{{\mathcal N}}

\newcommand{\bmo}{{\rm BMO}}
\newcommand{\bmoa}{{\rm BMOA}}

\newcommand{\f}{\frac}
\newcommand{\ov}{\overline}

\newcommand{\ze}{\zeta}

\newcommand{\ph}{\varphi}

\newcommand{\const}{\text{\rm const}}

\numberwithin{equation}{section}

\title[Nevanlinna and Smirnov classes: factoring the derivatives]
{Factoring derivatives of functions\\ 
in the Nevanlinna and Smirnov classes}
\author{Konstantin M. Dyakonov}
\address{ICREA and Universitat de Barcelona, Departament de Matem\`atica 
Aplicada i An\`alisi, Gran Via 585, E-08007 Barcelona, Spain}
\email{konstantin.dyakonov@icrea.es}
\thanks{Supported in part by grant MTM2008-05561-C02-01 from El Ministerio de Ciencia 
e Innovaci\'on (Spain) and grant 2009-SGR-1303 from AGAUR (Generalitat de Catalunya).}

\begin{document}
\begin{abstract}
We prove that, given a function $f$ in the Nevanlinna class $\cN$ and a positive integer 
$n$, there exist $g\in\cN$ and $h\in\bmoa$ such that $f^{(n)}=gh^{(n)}$. We may choose $g$ 
to be zero-free, so it follows that the zero sets for the class $\cN^{(n)}:=\{f^{(n)}:f\in\cN\}$ 
are the same as those for $\bmoa^{(n)}$. Furthermore, while the set of all products 
$gh^{(n)}$ (with $g$ and $h$ as above) is strictly larger than $\cN^{(n)}$, 
we show that the gap is not too large, at least when $n=1$. Precisely speaking, the class 
$\{gh':g\in\cN,\,h\in\bmoa\}$ turns out to be the smallest ideal space containing $\{f':f\in\cN\}$, 
where \lq\lq ideal" means invariant under multiplication by $H^\infty$ functions. Similar 
results are established for the Smirnov class $\cN^+$. 
\end{abstract}

\maketitle

\section{Introduction and results} 

Let $\mathcal H(\D)$ stand for the set of holomorphic functions on the disk $\D:=\{z\in\C:|z|<1\}$. 
Given a class $X\subset\mathcal H(\D)$ and an integer $n\in\N:=\{1,2,\dots\}$, we write 
$$X^{(n)}:=\{f^{(n)}:\,f\in X\},$$ 
where $f^{(n)}$ is the $n$th derivative of $f$. When $n=1$, we also use the notation $X'$ instead 
of $X^{(1)}$. Further, we denote by $\mathcal Z(X)$ the collection of zero sets for $X$; 
a (discrete) subset $E$ of $\D$ will thus belong to $\mathcal Z(X)$ 
if and only if $E=\{z\in\D:f(z)=0\}$ for some non-null function $f\in X$. 
Now, if $X$ and $Y$ are subclasses of $\mathcal H(\D)$, we put 
$$X\cdot Y:=\{fg:\,f\in X,\,g\in Y\}.$$ 
Finally, a vector space $X$ contained in $\mathcal H(\D)$ is said to be {\it ideal} if 
$$H^\infty\cdot X\subset X,$$ 
where, as usual, $H^\infty$ is the space of bounded holomorphic functions on $\D$. 

\par Our starting point is a result of Cohn and Verbitsky \cite{CV} which asserts, or rather implies, that 
\begin{equation}\label{eqn:factcv}
\left(H^p\right)^{(n)}=H^p\cdot\bmoa^{(n)}
\end{equation}
whenever $0<p<\infty$ and $n\in\N$. Here, we write $H^p$ for the classical (holomorphic) {\it Hardy spaces} on the 
disk, and $\bmoa$ for the \lq\lq analytic subspace" of $\bmo=\bmo(\T)$, the space of functions with {\it bounded mean 
oscillation} on the circle $\T:=\partial\D$. More precisely, $\bmoa$ can be defined as $H^1\cap\bmo$; as to the 
definitions of (and background information on) $H^p$ and $\bmo$, the reader will find these standard matters in 
\cite[Chapters II and VI]{G}. 
\par For $n=1$, identity \eqref{eqn:factcv} appeared in Cohn's earlier paper \cite{C}. On the other hand, \cite{CV} 
extends \eqref{eqn:factcv} to the case of a fractional derivative and still further; indeed, more general factorization 
theorems involving tent spaces -- and Triebel spaces -- are actually established there. It is also shown in \cite{CV} 
that, when factoring $f^{(n)}$ (for $f\in H^p$) in the sense of \eqref{eqn:factcv}, one may choose the $H^p$ factor 
on the right to be an outer function. As a consequence, one sees that 
\begin{equation}\label{eqn:zeroshp}
\mathcal Z\left(\left(H^p\right)^{(n)}\right)=\mathcal Z\left(\bmoa^{(n)}\right).
\end{equation}
In particular, for any fixed $n$, the zero sets for $\left(H^p\right)^{(n)}$ are the same for all $p\in(0,\infty)$. This 
last fact was contrasted in \cite{CV} with the Bergman space situation (where, by \cite{Ho}, different $A^p$ spaces 
have different zero sets). We wish to add, in this connection, that a similar Bergman-type phenomenon (different 
zero sets for different $p$'s) is also encountered in certain \lq\lq small" $H^p$-related spaces; namely, it occurs 
\cite{DJAM} for the {\it star-invariant subspaces} $H^p\cap\theta\,\ov{H^p_0}$ associated with an inner function $\theta$. 

\par Also related to \eqref{eqn:factcv} is Aleksandrov and Peller's work \cite{AP} that dealt with the case 
of $p\in[1,\infty)$ and $n=1$. There, for a given $f\in H^p$, a {\it weak factorization} $f'=\sum_{j=1}^4g_jh'_j$ 
was constructed with suitable $g_j\in H^p$ and $h_j\in H^\infty$. Yet another weak factorization theorem from \cite{AP}, 
which establishes a connection between $\bmoa'$ and $(H^\infty)'$, will be employed in Section 4 below. 

\par The purpose of this paper is to find out whether -- and/or to which extent -- the (strong) factorization 
theorem \eqref{eqn:factcv} carries over to the {\it Nevanlinna class} $\cN$, or the {\it Smirnov class} $\cN^+$, 
in place of $H^p$. 

\par Let us recall that $\cN$ is defined as the set of functions $f\in\mathcal H(\D)$ with 
$$\sup_{0<r<1}\int_\T\log^+|f(r\ze)|\,|d\ze|<\infty,$$ 
while $\cN^+$ is formed by those $f\in\cN$ which satisfy 
$$\lim_{r\to1^-}\int_\T\log^+|f(r\ze)|\,|d\ze|=\int_\T\log^+|f(\ze)|\,|d\ze|.$$ 
Equivalently, the elements of $\cN$ (resp., $\cN^+$) are precisely the ratios $u/v$, with $u,v\in H^\infty$ and 
with $v$ nonvanishing (resp., outer) on $\D$; for this and other characterizations of the two classes, see 
\cite[Chapter II]{G}. 
\par As far as factorization theorems of the form \eqref{eqn:factcv} are concerned, we can hardly expect the 
behavior of $\cN$ or $\cN^+$ to mimic that of $H^p$ too closely. In fact, as we shall soon explain, it is 
the \lq\lq easy" part of \eqref{eqn:factcv}, i.\,e., the inclusion 
\begin{equation}\label{eqn:easyfactcv}
\left(H^p\right)^{(n)}\supset H^p\cdot\bmoa^{(n)}
\end{equation}
that admits no extension to the Nevanlinna or Smirnov setting. Meanwhile, we remark 
that \eqref{eqn:easyfactcv} is indeed easy to deduce, at least for $p=2$, from the 
(not so easy, but readily available) descriptions of $\left(H^p\right)^{(n)}$ and 
$\bmoa^{(n)}$ as the appropriate Triebel spaces; see \cite{T}. 
One of these tells us that, for $\ph\in\mathcal H(\D)$, 
$$\ph\in\left(H^p\right)^{(n)}\iff\int_\T\left(\int_0^1|\ph(r\ze)|^2(1-r)^{2n-1}dr\right)^{p/2}|d\ze|<\infty$$
for all $n\in\N$ and $0<p<\infty$, a fact that has no counterpart for $\cN$ or $\cN^+$. 
The other, which involves a Carleson measure characterization of $\bmoa$, will be mentioned in Section 2 below. 
\par Now, to see that the $\cN$ and $\cN^+$ versions of \eqref{eqn:easyfactcv} actually break down, already for $n=1$, 
we recall Hayman's and Yanagihara's results from \cite{Ha, Y} saying that neither $\cN$ nor $\cN^+$ is invariant with 
respect to integration. More explicitly, Hayman \cite{Ha} gave an example of a function $f\in\cN$ whose antiderivative 
$F(z):=\int_0^zf(\ze)d\ze$ is not in $\cN$; Yanagihara \cite{Y} then strengthened this by showing that $F$ need not be 
in $\cN$ even for $f\in\cN^+$. Consequently, $\cN^+$ is not contained in $\cN'$, whence {\it a fortiori} 
\begin{equation}\label{eqn:noncont1}
\cN\not\subset\cN'\quad\text{\rm and}\quad\cN^+\not\subset(\cN^+)'. 
\end{equation}
Since $\cN\cdot\bmoa'$ (resp., $\cN^+\cdot\bmoa'$) contains $\cN$ (resp., $\cN^+$), we readily deduce 
from \eqref{eqn:noncont1} that 
\begin{equation}\label{eqn:noncont2}
\cN\cdot\bmoa'\not\subset\cN'\quad\text{\rm and}\quad\cN^+\cdot\bmoa'\not\subset(\cN^+)'.
\end{equation}
A similar conclusion holds for higher order derivatives as well. 
\par We prove, however, that the \lq\lq difficult" part of \eqref{eqn:factcv}, i.\,e., the inclusion 
\begin{equation}\label{eqn:diffactcv}
\left(H^p\right)^{(n)}\subset H^p\cdot\bmoa^{(n)}
\end{equation}
does remain valid with either $\cN$ or $\cN^+$ in place of $H^p$. 

\begin{thm}\label{thm:factnevsmi} For each $n\in\N$, we have 
$$\cN^{(n)}\subset\cN\cdot\bmoa^{(n)}\quad\text{and}
\quad(\cN^+)^{(n)}\subset\cN^+\cdot\bmoa^{(n)}.$$ 
More precisely, given $f\in\cN$ (resp., $f\in\cN^+$), one can find a zero-free function $g\in\cN$ 
(resp., an outer function $g\in\cN^+$) and an $h\in\bmoa$ such that $f^{(n)}=gh^{(n)}$. 
\end{thm}

\par It should be mentioned that our method also applies to the meromorphic Nevanlinna class $\cN_{\rm mer}$, 
defined as the set of quotients $u/v$, where $u,v\in H^\infty$ and $v$ is merely required to be non-null. 
Moreover, a glance at our proof of Theorem \ref{thm:factnevsmi} will reveal that if the original function $f$ 
is of the form $F/I$, with $F\in\cN^+$ and $I$ inner, then we may take $g=G/I^{n+1}$, with $G$ outer. 
And again, just as in the $H^p$ setting, our factorization theorem yields information on the zero sets. 

\begin{cor}\label{cor:zerosnev} We have 
$$\mathcal Z\left(\cN^{(n)}\right)=\mathcal Z\left(\bmoa^{(n)}\right),\qquad n\in\N.$$
\end{cor} 

\par Indeed, Theorem \ref{thm:factnevsmi} shows that every zero set for $\cN^{(n)}$ is a zero set 
for $\bmoa^{(n)}$, while the converse is immediate from the fact that $\bmoa\subset\cN$. 
Furthermore, since $\cN^+$ lies between $\bmoa$ and $\cN$, as does every $H^p$ with $0<p<\infty$, 
Corollary \ref{cor:zerosnev} obviously implies the identity 
$$\mathcal Z\left((\cN^+)^{(n)}\right)=\mathcal Z\left(\bmoa^{(n)}\right)$$ 
and also \eqref{eqn:zeroshp}. 

\par Finally, restricting ourselves to the case $n=1$, we wish to take a closer look at the inclusion 
$$\cN'\subset\cN\cdot\bmoa'$$ 
from Theorem \ref{thm:factnevsmi}, along with its $\cN^+$ counterpart. We know from \eqref{eqn:noncont2} that 
the inclusion is proper, and we now stress an important point of distinction between the two sides. Namely, the 
right-hand side, $\cN\cdot\bmoa'$, is ideal (i.\,e., invariant under multiplication by $H^\infty$ functions), 
whereas the left-hand side, $\cN'$, is not. Moreover, the space $\cN'$ is {\it highly nonideal} in the sense that 
even the identity function $z$ is not a multiplier thereof! (Otherwise, the formula 
$$g=(zg)'-zg',\qquad g\in\cN,$$ 
would imply that $\cN$ is contained in $\cN'$, which we know is false.) A similar remark applies to $(\cN^+)'$. 

\par Our last result states, then, that $\cN\cdot\bmoa'$ is actually the smallest ideal space containing $\cN'$, 
and that the same is true in the $\cN^+$ setting. 

\begin{thm}\label{thm:idealhull} {\rm (a)} The class $\cN\cdot\bmoa'$ is the ideal hull of $\cN'$. In other words, 
$\cN\cdot\bmoa'$ is an ideal vector space that contains $\cN'$ and is contained in every ideal space $X$ with 
$\cN'\subset X$. 
\par {\rm (b)} Similarly, $\cN^+\cdot\bmoa'$ is the ideal hull of $(\cN^+)'$. 
\end{thm}

Now let us turn to the proofs. 

\section{Preliminaries}

A couple of lemmas will be needed. 

\begin{lem}\label{lem:factderbmoa} Let $k\ge0$ and $l\ge1$ be integers. If $\ph\in\bmoa^{(l)}$ and 
$\psi$ is a function in $\mathcal H(\D)$ satisfying 
\begin{equation}\label{eqn:psibigoh}
\psi(z)=O((1-|z|)^{-k}),\qquad z\in\D, 
\end{equation}
then $\ph\psi\in\bmoa^{(k+l)}$.
\end{lem}

\begin{proof} It is known (see, e.\,g., \cite{J, Sh, T}) that a function $F\in\mathcal H(\D)$ 
will be in $\bmoa^{(n)}$, with $n\in\N$, if and only if 
the measure $|F(z)|^2(1-|z|)^{2n-1}dx\,dy$ (where $z=x+iy$) is a Carleson measure. 
The required result follows from this immediately, since \eqref{eqn:psibigoh} yields 
$$|\ph(z)\psi(z)|^2(1-|z|)^{2(k+l)-1}\le\const\cdot|\ph(z)|^2(1-|z|)^{2l-1}$$ 
for all $z\in\D$. 
\end{proof}

When $k=0$, the above lemma reduces to saying that 
\begin{equation}\label{eqn:idbmoan}
H^\infty\cdot\bmoa^{(n)}\subset\bmoa^{(n)}
\end{equation}
for all $n\in\N$; in other words, $\bmoa^{(n)}$ is an ideal space. This in turn leads to the next observation. 

\begin{lem}\label{lem:linear} For each $n\in\N$, the sets $\cN\cdot\bmoa^{(n)}$ and $\cN^+\cdot\bmoa^{(n)}$ 
are ideal vector spaces. 
\end{lem}

\begin{proof} It is clear that the two sets are invariant under multiplication by $H^\infty$ functions, 
but maybe not quite obvious that they are vector spaces. It is the linearity property 
$$f_1,f_2\in\cN\cdot\bmoa^{(n)}\implies f_1+f_2\in\cN\cdot\bmoa^{(n)}$$ 
(and a similar fact with $\cN^+$ in place of $\cN$) that should be verified. To this end, we write 
$$f_j=\f{u_j}{v_j}\cdot w_j^{(n)}\qquad(j=1,2),$$
where $u_j,v_j\in H^\infty$ and $w_j\in\bmoa$, and where $v_j$ is zero-free (resp., outer if the $f_j$'s 
are from $\cN^+\cdot\bmoa^{(n)}$). Note that 
$$f_1+f_2=\f1{v_1v_2}\cdot\left(u_1v_2w_1^{(n)}+u_2v_1w_2^{(n)}\right).$$ 
The two terms in brackets, and hence their sum, will be in $\bmoa^{(n)}$ by virtue of \eqref{eqn:idbmoan}, 
while the factor $1/(v_1v_2)$ will be in $\cN$ (resp., in $\cN^+$). 
\end{proof}

\section{Proof of Theorem \ref{thm:factnevsmi}}

We treat the case of $\cN$ first. Take $f\in\cN$ and write $f=u/v$, where $u,v\in H^\infty$ 
and $v$ has no zeros in $\D$. We have then 
\begin{equation}\label{eqn:leibniz}
f^{(n)}=\sum_{k=0}^n{n\choose k}u^{(n-k)}(1/v)^{(k)}. 
\end{equation}
For each $k\in\{0,\dots,n\}$, Fa\`a di Bruno's formula (see \cite[Chapter 3]{Sch}) yields 
\begin{equation}\label{eqn:faa} 
\left(\f 1v\right)^{(k)}=\sum C(m_1,\dots,m_k)\,v^{-m_1-\dots-m_k-1}
\prod_{j=1}^k\left(v^{(j)}\right)^{m_j}, 
\end{equation}
where the sum is over the $k$-tuples $(m_1,\dots,m_k)$ of nonnegative integers satisfying 
\begin{equation}\label{eqn:multi}
m_1+2m_2+\dots+km_k=k
\end{equation}
and where 
$$C(m_1,\dots,m_k)=(-1)^{m_1+\dots+m_k}\f{(m_1+\dots+m_k)!}{m_1!\dots m_k!}
\f{k!}{1!^{m_1}\dots k!^{m_k}}.$$ 
For any fixed multiindex $(m_1,\dots,m_k)$ as above, we clearly have 
\begin{equation}\label{eqn:conyo}
v^{-m_1-\dots-m_k-1}=v^{-n-1}\cdot v^{n-m_1-\dots-m_k}, 
\end{equation}
the last factor on the right being bounded. Indeed, 
\begin{equation}\label{eqn:polla}
v^{n-m_1-\dots-m_k}\in H^\infty, 
\end{equation}
since it follows from \eqref{eqn:multi} that $n-m_1-\dots-m_k\ge0$. We further observe that, for $j\in\N$, 
\begin{equation}\label{eqn:culo}
v^{(j)}(z)=O((1-|z|)^{-j}),\qquad z\in\D
\end{equation}
(because $v\in H^\infty$), and this implies together with \eqref{eqn:multi} that 
\begin{equation}\label{eqn:growth}
\prod_{j=1}^k\left[v^{(j)}(z)\right]^{m_j}=O((1-|z|)^{-k}),\qquad z\in\D.
\end{equation}
\par Combining \eqref{eqn:faa} and \eqref{eqn:conyo}, we see that the $k$th summand in \eqref{eqn:leibniz} takes 
the form $v^{-n-1}w_k$, where 
\begin{equation}\label{eqn:pene}
w_k:={n\choose k}\sum C(m_1,\dots,m_k)\,u^{(n-k)}v^{n-m_1-\dots-m_k}\prod_{j=1}^k\left(v^{(j)}\right)^{m_j}; 
\end{equation}
the sum is understood as in \eqref{eqn:faa}. We want to show that $w_k\in\bmoa^{(n)}$, and our plan is to check 
the corresponding inclusion for each individual term in \eqref{eqn:pene}. Thus, we claim that the function 
$$\Phi_{m_1,\dots,m_k}:=u^{(n-k)}v^{n-m_1-\dots-m_k}\prod_{j=1}^k\left(v^{(j)}\right)^{m_j}$$ 
satisfies 
\begin{equation}\label{eqn:phi}
\Phi_{m_1,\dots,m_k}\in\bmoa^{(n)}
\end{equation}
whenever $0\le k\le n$ and the $m_j$'s are related by \eqref{eqn:multi}. 
\par First let us verify \eqref{eqn:phi} in the case $k\le n-1$. To this end, we notice that 
$$u^{(n-k)}\in(H^\infty)^{(n-k)}\subset\bmoa^{(n-k)},$$ 
where $n-k\ge1$, while 
$$[v(z)]^{n-m_1-\dots-m_k}\prod_{j=1}^k\left[v^{(j)}(z)\right]^{m_j}=O((1-|z|)^{-k}),\qquad z\in\D,$$ 
by virtue of \eqref{eqn:polla} and \eqref{eqn:growth}. The validity of \eqref{eqn:phi} is then guaranteed by 
Lemma \ref{lem:factderbmoa}. 
\par Now if $k=n$, then the multiindices involved are of the form $(m_1,\dots,m_n)$ with 
$\sum_{j=1}^njm_j=n$. For any such multiindex, at least one of the $m_j$'s (say, $m_l$ with an $l\in\{1,\dots,n\}$) 
must be nonzero, so that $m_l\ge1$ and 
\begin{equation}\label{eqn:rape}
l(m_l-1)+\sum_{1\le j\le n,\,j\ne l}jm_j=n-l.
\end{equation}
Consider the factorization 
$$\Phi_{m_1,\dots,m_n}=v^{(l)}\cdot\left\{uv^{n-m_1-\dots-m_n}\left(v^{(l)}\right)^{m_l-1}
\prod_{1\le j\le n,\,j\ne l}\left(v^{(j)}\right)^{m_j}\right\}.$$ 
The first factor, $v^{(l)}$, is then in $(H^\infty)^{(l)}$ and hence in $\bmoa^{(l)}$, while the second factor 
(the one in curly brackets) is $O((1-|z|)^{-n+l})$. The latter estimate is due to \eqref{eqn:culo} and 
\eqref{eqn:rape}, coupled with the fact that $u$ and $v$ are in $H^\infty$. Applying Lemma \ref{lem:factderbmoa} 
to the current factorization, we arrive at \eqref{eqn:phi}, this time with $k=n$. 
\par Now that \eqref{eqn:phi} is known to be true, we infer that the functions $w_k$ from \eqref{eqn:pene} are 
all in $\bmoa^{(n)}$, whence obviously $\sum_{k=0}^nw_k\in\bmoa^{(n)}$. Recalling that 
$$f^{(n)}=v^{-n-1}\sum_{k=0}^nw_k,$$ 
we finally conclude that $f^{(n)}$ can be written as $gh^{(n)}$, where $g:=v^{-n-1}\in\cN$ and $h$ is a 
$\bmoa$ function satisfying $h^{(n)}=\sum_{k=0}^nw_k$. 
\par The case of $\cN^+$ is similar. This time, $v$ is taken to be an {\it outer} function in $H^\infty$, so 
$g=v^{-n-1}$ will be an outer function in $\cN^+$.\quad\qed

\section{Proof of Theorem \ref{thm:idealhull}} 

We shall only prove (a), the proof of (b) being similar. 
We know from Lemma \ref{lem:linear} that $\cN\cdot\bmoa'$ is an ideal space. 
Furthermore, Theorem \ref{thm:factnevsmi} tells us that $\cN\cdot\bmoa'$ 
contains $\cN'$. It remains to verify that, whenever $X$ is an ideal space with $\cN'\subset X$, 
we necessarily have 
\begin{equation}\label{eqn:gadina}
\cN\cdot\bmoa'\subset X.
\end{equation}
\par Take any $g\in\cN$ and $h\in H^\infty$. Note that 
\begin{equation}\label{eqn:ghprime}
gh'=(gh)'-g'h,
\end{equation}
where both terms on the right are in $X$. Indeed, $(gh)'$ is obviously in $\cN'$ and hence in $X$, while the 
inclusion $g'h\in X$ is due to the facts that $g'\in\cN'\subset X$ and $hX\subset X$ (recall that $X$ is ideal). 
It now follows from \eqref{eqn:ghprime} that $gh'\in X$, and we have thereby checked that 
\begin{equation}\label{eqn:cont}
\cN\cdot(H^\infty)'\subset X. 
\end{equation}
\par Finally, given $\eta\in\bmoa$, we invoke a result of Aleksandrov and Peller \cite[Theorem 3.4]{AP} 
to find functions $\varphi_j,\psi_j\in H^\infty$ 
($j=1,2$) such that $\eta'=\varphi_1\psi'_1+\varphi_2\psi'_2$. Letting $g\in\cN$ as before, we get 
\begin{equation}\label{eqn:geta}
g\eta'=g\varphi_1\psi'_1+g\varphi_2\psi'_2.
\end{equation}
Here, the two terms of the form $g\varphi_j\psi'_j$ are in $\cN\cdot(H^\infty)'$, so we infer 
from \eqref{eqn:cont} that they are also in $X$. The right-hand side of \eqref{eqn:geta} is therefore 
in $X$, and so is the left-hand side, $g\eta'$. Thus we conclude that $g\eta'\in X$ for all 
$g\in\cN$ and $\eta\in\bmoa$. This establishes \eqref{eqn:gadina} and completes the proof.\quad\qed


\begin{thebibliography}{12}

\bibitem{AP} A. B. Aleksandrov and V. V. Peller, {\it Hankel operators and similarity 
to a contraction}, Internat. Math. Res. Notices \textbf{1996}, No. 6, 263--275. 

\bibitem{C} W. S. Cohn, {\it A factorization theorem for the derivative of a function 
in $H^p$}, Proc. Amer. Math. Soc. \textbf{127} (1999), 509--517. 

\bibitem{CV} W. S. Cohn and I. E. Verbitsky, {\it Factorization of tent spaces and Hankel 
operators}, J. Funct. Anal. \textbf{175} (2000), 308--329. 

\bibitem{DJAM} K. M. Dyakonov, {\it Zero sets and multiplier theorems for star-invariant 
subspaces}, J. Anal. Math. \textbf{86} (2002), 247--269. 

\bibitem{G} J. B. Garnett, {\it Bounded analytic functions}, Revised first edition, 
Springer, New York, 2007. 

\bibitem{Ha} W. K. Hayman, {\it On the characteristic of functions meromorphic in the unit 
disk and of their integrals}, Acta Math. \textbf{112} (1964), 181--214. 

\bibitem{Ho} C. Horowitz, {\it Zeros of functions in the Bergman spaces}, Duke Math. J. 
\textbf{41} (1974), 693--710. 

\bibitem{J} M. Jevti\'c, {\it On the Carleson measure characterization of $\bmoa$ functions 
on the unit ball}, Proc. Amer. Math. Soc. \textbf{114} (1992), 379--386. 

\bibitem{Sch} L. Schwartz, {\it Cours d'analyse}, Second edition, Hermann, Paris, 1981. 

\bibitem{Sh} F. A. Shamoyan, {\it Toeplitz operators in some spaces of holomorphic functions 
and a new characterization of the class $\bmo$}, Izv. Akad. Nauk Armyan. SSR Ser. Mat. 
\textbf{22} (1987), no. 2, 122--132. 

\bibitem{T} H. Triebel, {\it Theory of function spaces. II}, Monographs in Mathematics, 84, 
Birkh\"auser Verlag, Basel, 1992. 

\bibitem{Y} N. Yanagihara, {\it On a class of functions and their integrals}, Proc. London 
Math. Soc. (3) \textbf{25} (1972), 550--576. 

\end{thebibliography}
\end{document}